\documentclass[preprint, review, 12pt]{elsarticle}

\usepackage{amsthm}
\usepackage{amssymb}
\usepackage{amsmath}
\usepackage{amsfonts}
\usepackage{mathrsfs}
\usepackage{color}
\usepackage{microtype}
\usepackage{fancyhdr}
\usepackage{caption}
\usepackage{wrapfig}
\usepackage{cases}
\usepackage{setspace}
\usepackage{geometry}
\usepackage{tikz} 
\usetikzlibrary{positioning, shapes.geometric} 
\usepackage[colorlinks,linkcolor=blue,anchorcolor=blue,citecolor=blue,CJKbookmarks=True]{hyperref}
\usepackage[american]{babel}
\usepackage{microtype}
\usepackage{ragged2e}

\allowdisplaybreaks[4]
\biboptions{numbers,sort&compress}

\newtheorem{lemma}{Lemma}[section]
\newtheorem{theorem}{Theorem}[section]

\newtheorem{remark}{Remark}[section]
\newtheorem{proposition}{Proposition}[section]

\numberwithin{equation}{section}
\allowdisplaybreaks[4]

\newcommand{\<}{\langle}
\renewcommand{\>}{\rangle}

\newcommand{\hR}{\mathbb R}

\newcommand{\rd}{\mathrm d}

\newcommand{\cI}{\mathcal I}

\journal{Physica D:\ Nonlinear Phenomena}

\begin{document}

\pagenumbering{arabic}

\begin{spacing}{1}

\begin{frontmatter}

\title{Density functions for the overdamped generalized Langevin equation and its Euler--Maruyama method:\ smoothness and convergence}

\author[ad1]{Xinjie Dai}
\ead{dxj@ynu.edu.cn} 

\author[ad2]{Diancong Jin\corref{cor1}}
\ead{jindc@hust.edu.cn} 

\address[ad1]{School of Mathematics and Statistics, Yunnan University, Kunming 650504, China}

\address[ad2]{School of Mathematics and Statistics \& Hubei Key Laboratory of Engineering Modeling and Scientific Computing, Huazhong University of Science and Technology, Wuhan 430074, China}

\cortext[cor1]{Corresponding author}
	
\begin{abstract}
This paper focuses on studying the convergence rate of the density function of the Euler--Maruyama (EM) method, when applied to the overdamped generalized Langevin equation with fractional noise which serves as an important model in many fields. Firstly, we give an improved upper bound estimate for the total variation distance between random variables by their Malliavin--Sobolev norms. Secondly, we establish the existence and smoothness of the density function for both the exact solution and the numerical one. Based on the above results, the convergence rate of the density function of the numerical solution is obtained, which relies on the regularity of the noise and kernel. This convergence result provides a powerful support for numerically capturing the statistical information of the exact solution through the EM method. 
\end{abstract}	


\end{frontmatter}

\thispagestyle{plain}

\section{Introduction}

The generalized Langevin equation (GLE) provides a precise description of coarse-grained variable dynamics in reduced dimension models \cite{LeiBakerLi2016} and has numerous applications in scientific fields such as nanoscale biophysics \cite{Kou2008}, viscoelastic fluids \cite{DidierNguyen2022}, machine learning \cite{XieCarE2024}, and so on. For example, in statistical physics, the position $x (t)$ of a moving particle with mass $m$ in the energy potential $V$ at time $t$ can be modelled by the GLE
\begin{align*}
m \ddot x(t) = - \nabla V(x(t)) - \int_0^t K(t-s) \dot x(s) \mathrm d s + \eta(t). 
\end{align*}
According to the fluctuation-dissipation theorem, the convolutional kernel $K(t)$ of the friction (dissipation) and the random force (fluctuation) $\eta(t)$ are associated through the relation 
\begin{align*}
\mathbb E \big[ \eta(t) \eta(s) \big] = k_{B} T_A K(t-s), \quad \mbox{for } s \leq t,
\end{align*}
where $k_{B}$ and $T_A$ are Boltzmann's constant and the absolute temperature, respectively. 

To capture the ubiquitous memory phenomena in biology and physics, the fluctuation $\eta(t)$ is often characterized by the fractional noise, and then the fluctuation-dissipation theorem reveals that the memory kernel $K(t)$ is proportional to a power law $t^{-\alpha}$ with some $\alpha \in (0,1)$. In the overdamped regime ($m\ll 1$), the GLE with fractional noise can be represented as a stochastic Volterra equation with weakly singular kernels as follows \cite{LiLiu2017}:\ 
\begin{align}\label{eq.GLE}
x(t) = x_0 + \frac{1}{\Gamma(\alpha)}\int_0^{t} (t-s)^{\alpha-1} f(x(s)) \mathrm{d}s + \frac{\sigma}{\Gamma(\alpha)} \int_0^t (t-s)^{\alpha-1} \mathrm{d}W_H(s), \quad t \in [0, T]. 
\end{align}
Here, $W_H$ is a fractional Brownian motion (fBm) with Hurst index $H$ on some complete filtered probability space $(\Omega,\mathscr F,\{\mathscr F_t\}_{t\in[0,T]}, \mathbb P)$ satisfying the usual condition, $\Gamma$ denotes the Gamma function, and $\sigma \neq 0$ is a constant. We always assume that the drift function $f: \mathbb R\rightarrow \mathbb R$ is Lipschitz continuous and the initial value $x_0\in\mathbb R$ is deterministic. In this setting, Eq.\ \eqref{eq.GLE} admits a unique strong solution for $H \in (1/2,1)$ and $\alpha \in (1-H,1)$; see Theorem 1 of \cite{LiLiu2017}. We also refer to \cite{LiLiu2019} for the existence and uniqueness of the limiting measure of the solution to Eq.\ \eqref{eq.GLE}. 

The solution to Eq.\ \eqref{eq.GLE} is essentially constituted of a family $\{x(t)\}_{t\in[0,T]}$ of random variables. All probabilistic information, such as expectations of important functionals of $\{x(t)\}_{t\in[0,T]}$ can be fully characterized by its density function. In terms of the density function for the exact solution, the existence and smoothness results have been obtained for other types of stochastic Volterra equations \cite{BesaluMarquez2021, Fan2015, FriesenJin2024}, for example, the standard Brownian motion (i.e., $H = 1/2$) case and the non-singular kernel (i.e., $\alpha > 1$) case. However, to the best of our knowledge, the existence and smoothness of the density function of the solution to Eq.\ \eqref{eq.GLE} are still unexplored, which is our first motivation.
In addition, note that it is extremely difficult to obtain a closed-form density function when the density function exists.
It is meaningful and indispensable to use the numerical method to generate the approximation of the density function of $\{x(t)\}_{t\in[0,T]}$. In numerical aspect, we are only aware that \cite{VojtaSkinner2019, VojtaWarhover2021} use the Euler--Maruyama (EM) method to simulate the density function of the solution of underdamped GLEs with reflecting and absorbing walls, respectively. The numerical research on the density function for Eq.\ \eqref{eq.GLE} remains scarce. Thus, in this paper, we also aim to fill the gap on numerically approximating the density function for the overdamped GLE \eqref{eq.GLE}. 

Our main contributions in this paper are threefold as follows:
\begin{itemize}
\item We give an improved upper bound estimate for the total variation distance between random variables by their Malliavin--Sobolev norms, in Proposition \ref{prop.D12}.
	
\item We establish the existence and smoothness of the density function for both the exact solution to Eq.\ \eqref{eq.GLE} and the numerical solution of the EM method; see Theorems \ref{thm.ExactDensityExis}--\ref{thm.EulerDensitySmooth}.

\item We obtain the convergence rate of the density function of the numerical solution generated by the EM method for Eq.\ \eqref{eq.GLE}, which will be performed in Theorem \ref{thm.DensityConvergence}. 
\end{itemize}
In Figure \ref{fig.FrameworkDiagram}, the relationship between the main results of this paper is listed for readability. Note that some efforts have been done in \cite{DaiHongShengZhou2023, DaiXiao2021, FangLi2020} to establish strong error estimates for the EM method of Eq.\ \eqref{eq.GLE}. 

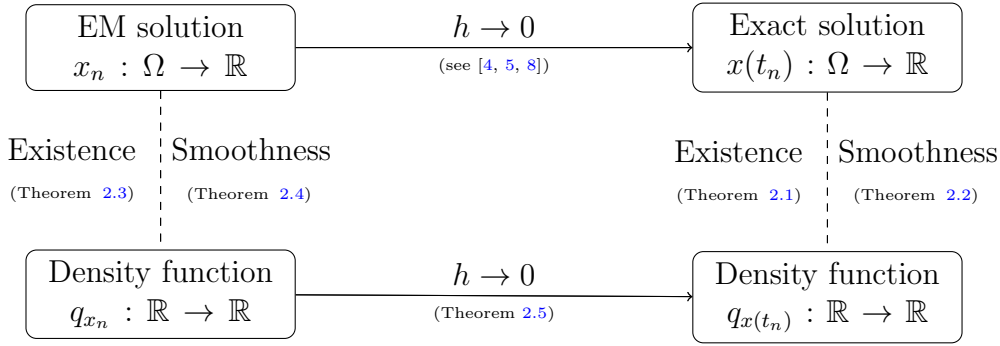
\begin{figure}[!ht]
\centering
\scalebox{1}{
\begin{tikzpicture}[node distance=60pt]
\node[draw, rounded corners, text width=8em, align=center] (UpperLeft) {EM solution $x_n: \Omega \rightarrow \hR$};

\node[draw, rounded corners, text width=8em, align=center, right=150pt of UpperLeft] (UpperRight) {Exact solution $x(t_n): \Omega \rightarrow \hR$};
\node[draw, rounded corners, text width=8em, align=center, below=of UpperLeft] (LowerLeft) {Density function $q_{x_n}: \hR \rightarrow \hR$};
\node[draw, rounded corners, text width=8em, align=center, below= of UpperRight] (LowerRight) {Density function $q_{x(t_n)}: \hR \rightarrow \hR$};

\draw[dashed] (UpperLeft) -- node[left, text width=5em, align=center] {Existence \tiny{(Theorem \ref{thm.EulerDensityExis})}} (LowerLeft);
\draw[dashed] (UpperLeft) -- node[right, text width=5em, align=center] {Smoothness \tiny{(Theorem \ref{thm.EulerDensitySmooth})}} (LowerLeft);

\draw[->] (UpperLeft) -- node[above] {$h \rightarrow 0$} (UpperRight);
\draw[->] (UpperLeft) -- node[below] {\tiny{(see \cite{DaiHongShengZhou2023, DaiXiao2021, FangLi2020})}} (UpperRight);

\draw[dashed] (UpperRight) -- node[left, text width=5em, align=center] {Existence \tiny{(Theorem \ref{thm.ExactDensityExis})}} (LowerRight);
\draw[dashed] (UpperRight) -- node[right, text width=5em, align=center] {Smoothness \tiny{(Theorem \ref{thm.ExactDensitySmooth})}} (LowerRight);

\draw[->] (LowerLeft) -- node[above] {$h \rightarrow 0$} (LowerRight);
\draw[->] (LowerLeft) -- node[below] {\tiny{(Theorem \ref{thm.DensityConvergence})}} (LowerRight);
\end{tikzpicture}
}
\caption{The relationship between the main results of this paper. Here, the convergence of $x_n$ and $q_{x_n}$ as $h \rightarrow 0$ is uniform for all $n \in \{1,2,\cdots,N\}$.} 
\label{fig.FrameworkDiagram}
\end{figure}

The rest of the paper is organized as follows. Section \ref{sec2} introduces the preliminaries including the Malliavin calculus and states the main results of this paper. Section \ref{sec3} provides several auxiliary lemmas. Section \ref{sec4} shows the detailed proofs of these theoretical findings. Throughout this paper, unless otherwise specified, we use the following notations. For the integer $m \geq 1$, denote by $C_b^m$ the space of not necessarily bounded real-valued functions that have continuous and bounded derivatives up to order $m$, and by $C_{b}^{\infty}$ the space of real-valued smooth functions whose all derivatives are bounded. We use $C$ as a generic constant and use $C(\cdot)$ if necessary to mention the parameters it depends on, whose values are always independent of the time step size $h$ and may differ at different occurrences.

\section{Preliminaries and main results}
\label{sec2}

\subsection{Preliminaries}

Let us first introduce some basic definitions and Malliavin calculus with respect to the fBm; see \cite{HongHuang2020, Nualart2006} for more details. When $H \in (\frac{1}{2},1)$, the covariance function of fBm $W_H$ satisfies 
\begin{align*}
{\rm Cov}(t,s) = H(2H-1) \int_{0}^{t}\int_{0}^{s} |u-v|^{2H-2} \mathrm{d}u \mathrm{d}v \quad \forall\, s,t \in [0,T].
\end{align*}
Denote by $\mathcal{E}$ the set of real-valued step functions on $[0,T]$ and let $\mathcal H$ be the Hilbert space defined as the closure of $\mathcal E$ with respect to the scalar product $\langle \mathbf{1}_{[0,t]}, \mathbf{1}_{[0,s]} \rangle_{\mathcal{H}} := {\rm Cov}(t,s)$. The mapping $\mathbf 1_{[0,t]}\mapsto W_H(t)$ can be extended to an isometry between $\mathcal{H}$ and a closed subspace of $L^2(\Omega;\mathbb{R})$. More precisely, denote this isometry by $\varphi \mapsto W_H(\varphi)$, then
\begin{align*}
\langle \varphi, \phi \rangle_{\mathcal{H}} = H(2H-1) \int_{0}^{T}\int_{0}^{T} \varphi(u) \phi(v) |u-v|^{2H-2} \mathrm{d}u \mathrm{d}v = \mathbb{E}\Big[ W_H(\varphi) W_H(\phi) \Big] \quad \forall\, \varphi,\phi \in \mathcal{H}.
\end{align*}

In fact, the covariance function also admits 
\begin{align*}
{\rm Cov}(t,s) = \int_{0}^{t \wedge s} \Phi_H (t,u) \Phi_H (s,u) \mathrm{d}u 
\end{align*}
with the kernel function
\begin{align*}
\Phi_H (t,s) := c_H s^{\frac12-H}\int_{s}^{t} (v-s)^{H-\frac32} v^{H-\frac12} \mathrm{d}v \, \mathbf 1_{\{s<t\}}. 
\end{align*}
Here, $c_H := \sqrt{\frac{H(2H-1)}{B(2-2H,H-\frac12)}}$ and $B(a, b) := \int_0^1 u^{a-1} (1-u)^{b-1} \mathrm{d}u$ with $a, b>0$ is the Beta function. Define the operator $K^*$ from $\mathcal{H}$ to $L^2([0,T];\mathbb{R})$ by
\begin{align*}
(K^*\varphi)(s):=\int_{s}^{T}\varphi(t) \frac{\partial}{\partial t} \Phi_H (t,s) \mathrm{d}t, \quad \varphi \in \mathcal{H}.
\end{align*}
Then, for any $\varphi,\, \phi \in \mathcal{H}$, 
\begin{align*}
\langle \varphi, \phi \rangle_{\mathcal{H}} = \langle K^*\varphi, K^*\phi \rangle_{L^2([0,T];\mathbb{R})}, 
\end{align*}
which implies that $K^*$ provides another isometry between $\mathcal{H}$ and a closed subspace of $L^2([0,T];\mathbb{R})$. 

Denote by $\mathcal{S}$ the class of smooth real-valued random variables such that $F \in \mathcal{S}$ has the form
\begin{align*}
F = g(W_H(\varphi_1),\ldots,W_H(\varphi_n)),
\end{align*}
where $g \in C_p^\infty(\mathbb{R}^n;\mathbb{R})$, $\varphi_i \in \mathcal{H}, \, i=1,\ldots,n,\, n \in \mathbb{N_+}$. Here, $C_p^\infty(\mathbb{R}^n;\mathbb{R})$ is the space of all real-valued smooth functions on $\mathbb{R}^n$ with all its partial derivatives growing polynomially. The Malliavin derivative with respect to the fBm of $F \in \mathcal{S}$ is an $\mathcal{H}$-valued random variable defined by
\begin{align*}
D F = \sum_{i=1}^n \frac{\partial}{\partial x_i} g(W_H(\varphi_1),\ldots,W_H(\varphi_n)) \varphi_i. 
\end{align*}
For any $p \geq1$, we denote the domain of $D$ in $L^p(\Omega;\mathbb{R})$ by $\mathbb{D}^{1,p}$, meaning that $\mathbb{D}^{1,p}$ is the closure of $\mathcal{S}$ with respect to the norm
\begin{align*}
\|F\|_{\mathbb{D}^{1,p}} = \Big( \mathbb{E} \big[ |F|^p \big] + \mathbb{E} \big[ \|DF\|_{\mathcal{H}}^p \big] \Big)^{\frac{1}{p}}.
\end{align*}
In a similar manner, for $F \in \mathcal{S}$, the iterated derivative $D^k F$ ($k \in \mathbb{N}_+$) is defined as a random variable with values in $\mathcal{H}^{\otimes k}$. For every $p \geq 1$ and $k \in \mathbb{N}_+$, denote by $\mathbb{D}^{k,p}$ the completion of $\mathcal{S}$ with respect to the norm
\begin{align*}
\|F\|_{\mathbb{D}^{k,p}} = \left( \mathbb{E}\Big[ |F|^p + \sum_{j=1}^{k} \|D^jF\|_{\mathcal{H}^{\otimes j}}^p \Big] \right)^{\frac{1}{p}}.
\end{align*}
We also denote $\mathbb{D}^{k,\infty} := \bigcap_{p \in [1,\infty)} \mathbb{D}^{k,p}$ and $\mathbb{D}^{\infty} := \bigcap_{k \geq 1} \mathbb{D}^{k,\infty}$ for simplicity. 

We proceed to introduce the adjoint operator $\delta$ of the derivative operator $D$, which is also known as the Skorohod integral. If an $\mathcal{H}$-valued random variable $\varphi\in L^2(\Omega;\mathcal{H})$ satisfies
\begin{align*}
\big| \mathbb{E}[\langle \varphi, DF \rangle_{\mathcal{H}}] \big| \leq C(\varphi) \| F \|_{L^2(\Omega;\mathbb{R})} \quad \forall\, F\in \mathbb{D}^{1,2},
\end{align*}
then $\varphi\in {\rm Dom}(\delta)$ and $\delta(\varphi)\in L^2(\Omega;\mathbb{R})$ is characterized by the dual formula
\begin{align} \label{eq.dualformu}
\mathbb{E}[ \langle \varphi, DF \rangle_{\mathcal{H}}] = \mathbb{E}[F\delta(\varphi)] \quad \forall\, F\in \mathbb{D}^{1,2}.
\end{align}
In particular, when $\varphi \in \mathcal{H}$ is deterministic, the Skorohod integral $\delta(\varphi)$ coincides with the Riemann--Stieltjes integral $\int_0^T \varphi(u) \mathrm{d} W_H(u)$. 

In what follows, by strengthening the conditions of \cite[Theorem 4.2]{NourdinPoly2013}, we obtain an improved upper bound for the total variation distance of random variables by their Malliavin--Sobolev norms, which plays a key role in proving Theorem \ref{thm.DensityConvergence}. 

\begin{proposition} \label{prop.D12}
Let $\{F_n\}_{n\ge1}$ be a sequence in $\mathbb{D}^{1,2}$ and $F_\infty\in\mathbb{D}^{2,4}$ with $\|DF_\infty\|_{\mathcal{H}} \geq c$ a.s.\ for some constant $c > 0$. If the laws of $\{F_n\}_{n\ge1}$ and $F_\infty$ are absolutely continuous with respect to the Lebesgue measure on $\mathbb{R}$, then the total variation distance between the laws of $F_{n}$ and $F_{\infty}$ satisfies
\begin{align*}
d_{\mathrm{TV}} ( F_{n}, F_{\infty} ) \leq C \left\| F_{n} - F_{\infty} \right\|_{\mathbb{D}^{1,2}} \quad \forall\, n \geq 1, 
\end{align*}
where the constant $C>0$ is independent of $n$. 
\end{proposition}

\begin{proof}
Let $A$ be a bounded Borel set of $\mathbb{R}$. For convenience, define
\begin{align*}
\phi(F_n,F_\infty) := \int_{F_\infty}^{F_n} \textbf{1}_A(x) \mathrm{d}x. 
\end{align*}
Then, it follows from the chain rule (see e.g., Proposition 2.3.8 of \cite{NourdinPeccati2012}) that
\begin{align*}
D\phi(F_n,F_\infty) = \textbf{1}_A(F_n) DF_n - \textbf{1}_A(F_\infty) DF_\infty, 
\end{align*}
which implies 
\begin{align} \label{eq.DL1_1}
&\quad\ \mathbb{E} \left[ \frac{\langle D \phi(F_n,F_\infty), DF_\infty \rangle_{\mathcal{H}} }{\|DF_\infty\|^2_{\mathcal{H}} }\right] \nonumber\\
&= \mathbb{E} \left[ \frac{\langle\textbf{1}_A(F_n) (DF_n-DF_\infty), DF_\infty\rangle_{\mathcal{H}} }{\|DF_\infty\|^2_{\mathcal{H}} } \right]
+\mathbb{E} \left[ \frac{\langle (\textbf{1}_A(F_n) - \textbf{1}_A(F_\infty)) DF_\infty, DF_\infty \rangle_{\mathcal{H}} }{\|DF_\infty\|^2_{\mathcal{H}} } \right] \nonumber\\
&= \mathbb{E} \left[ \frac{\langle \textbf{1}_A(F_n) (DF_n-DF_\infty), DF_\infty\rangle_{\mathcal{H}} }{\|DF_\infty\|^2_{\mathcal{H}} } \right]
 + \mathbb{E} \left[ \textbf{1}_A(F_n) - \textbf{1}_A(F_\infty) \right].
\end{align}
By Proposition 1.3.3 of \cite{Nualart2006}, one also has 
\begin{align} \label{eq.DL1_2}
\left\langle D \phi(F_n,F_\infty), D F_{\infty} \right\rangle_{\mathcal{H}} 
&= - \delta \left( \phi(F_n,F_\infty) D F_{\infty} \right) + \delta \left( D F_{\infty} \right) \phi(F_n,F_\infty).
\end{align}
Thus, it follows from \eqref{eq.DL1_1}, \eqref{eq.DL1_2} and the triangle inequality that 
\begin{align}\label{DL2}
\left| \mathbb{E}\left[\textbf{1}_A(F_n)-\textbf{1}_A(F_\infty)\right]\right| 
&\leq \left| -\mathbb{E}\left[\frac{ \delta \left( \phi(F_n,F_\infty) D F_{\infty} \right) }{\|DF_\infty\|^2_{\mathcal{H}}}\right]\right|+\left|\mathbb{E}\left[\frac{\delta \left( D F_{\infty} \right) \phi(F_n,F_\infty) }{\|DF_\infty\|^2_{\mathcal{H}}}\right]\right| \nonumber\\
&\quad\, + \left|-\mathbb{E}\left[\frac{ \langle \textbf{1}_A(F_n) (DF_n-DF_\infty), DF_\infty \rangle_{\mathcal{H}}}{\|DF_\infty\|^2_{\mathcal{H}}}\right]\right|.
\end{align}
Using the dual formula \eqref{eq.dualformu}, H\"older inequality, $\|DF_\infty\|^{-1}_{\mathcal{H}} \leq c^{-1}$ and $F_\infty \in \mathbb{D}^{2,4}$ shows 
\begin{align*}
\left|-\mathbb{E}\left[\frac{\delta\left(\phi(F_n,F_\infty) D F_{\infty} \right)}{\|DF_\infty\|^2_{\mathcal{H}}}\right]\right| 
&= \left| \mathbb{E} \left[ \left\< D \frac{1}{\| DF_\infty \|_{\mathcal{H}}^2}, D F_{\infty} \phi(F_n,F_\infty) \right\>_{\mathcal{H}} \right]\right| \\ 
&= \left|\mathbb{E}\left[\frac{2\left \langle D^2 F_{\infty}, D F_{\infty} \otimes D F_{\infty}\right\rangle_{\mathcal{H}^{\otimes 2}}}{\left\|D F_{\infty}\right\|_{\mathcal{H}}^{4}} \int_{F_{\infty}}^{F_{n}} \mathbf{1}_{A}(x) \mathrm{d} x\right]\right| \\
&\leq C\|F_n-F_\infty\|_{L^2(\Omega)}. 
\end{align*}
Similarly, the second term in the right hand of $\eqref{DL2}$ is also bounded by $C\|F_n-F_\infty\|_{L^2(\Omega)}$, while the third one is bounded by $C\|DF_n-DF_\infty\|_{L^2(\Omega;\mathcal{H})}$. Therefore, 
\begin{align*}
d_{\mathrm{TV}} ( F_{n}, F_{\infty} ) 
&= \sup_{A\in\mathcal{B}(\mathbb{R})} \left|P\left(F_{n} \in A\right) - P\left(F_{\infty} \in A\right)\right| \\
&= \sup_{A\in\mathcal{B}(\mathbb{R})} \left|\mathbb{E} \left[\textbf{1}_A(F_n)-\textbf{1}_A(F_\infty) \right] \right|\\ 
&\leq C \left\| F_{n} - F_{\infty} \right\|_{\mathbb{D}^{1,2}}, 
\end{align*}
which completes the proof. 
\end{proof}

\subsection{Main results}

Now, we fix $H \in (1/2, 1)$ and $\alpha \in (1-H,1)$ for the well-posedness of Eq.\ \eqref{eq.GLE}; see Theorem 1 of \cite{LiLiu2017}. We first present the existence and smooth of the density function of the exact solution to Eq.\ \eqref{eq.GLE}, respectively, as stated in Theorems \ref{thm.ExactDensityExis} and \ref{thm.ExactDensitySmooth}. 

\begin{theorem} \label{thm.ExactDensityExis}
Let $t \in (0,T]$ and $f \in C_{b}^{1}$. Then the law of the exact solution $x(t)$ is absolutely continuous with respect to the Lebesgue measure on $\mathbb{R}$.
\end{theorem}

\begin{theorem} \label{thm.ExactDensitySmooth}
Let $t \in (0,T]$ and $f \in C_{b}^{\infty}$. Then the exact solution $x(t)$ admits an infinitely differentiable density function.
\end{theorem}

For a fixed integer $N \geq 1$, let $\{ t_n := n h, \, n= 0, 1, \cdots, N \}$ be a uniform partition of $[0,T]$ with the time step size $h := T/N$. As introduced in \cite{FangLi2020}, the EM method for Eq.\ \eqref{eq.GLE} can be formulated as
\begin{align}\label{eq.EM}
x_n = x_0 + \frac{1}{\Gamma(\alpha)} \sum_{j=1}^{n} f(x_{j-1}) \int_{t_{j-1}}^{t_j} (t_n-s)^{\alpha-1} \mathrm{d}s + G(t_n), \quad n = 1,2,\cdots,N,
\end{align}
where 
\begin{align}\label{eq.Gt}
G(t) := \frac{\sigma}{\Gamma(\alpha)} \int_0^{t} (t-s)^{\alpha-1} \mathrm{d}W_H(s) \quad \forall\, t\in[0,T].
\end{align}

Then, we are able to state the main results for the EM method \eqref{eq.EM} of Eq.\ \eqref{eq.GLE}. 

\begin{theorem} \label{thm.EulerDensityExis}
Let $n\in\{1,2,\cdots,N\}$ and $f \in C_{b}^{1}$. Then the law of the numerical solution $x_n$ is absolutely continuous with respect to the Lebesgue measure on $\mathbb{R}$.
\end{theorem}

\begin{theorem} \label{thm.EulerDensitySmooth}
Let $n\in\{1,2,\cdots,N\}$ and $f \in C_{b}^{\infty}$. Then the numerical solution $x_n$ admits an infinitely differentiable density function.
\end{theorem}

\begin{theorem} \label{thm.DensityConvergence}
If $f \in C_{b}^{2}$, then there exists some positive constant $C$ independent of $h$ such that 
\begin{align*}
\| q_{x(t_n)} - q_{x_n} \|_{L^1(\mathbb{R};\mathbb{R})} \leq C h^{\alpha+H-1} \quad \forall \, n \in \{1,2,\cdots,N\}, 
\end{align*}
where $q_{x(t_n)}$ and $q_{x_n}$ denote the probability density functions of $x(t_n)$ and $x_n$, respectively.
\end{theorem}

\section{Auxiliary lemmas}
\label{sec3}

The following two lemmas are taken from Lemma 4.1 of \cite{ChenCuiHongSheng2023} and Theorem 3.1 of \cite{DaiHongShengZhou2023}, respectively. 

\begin{lemma} \label{Faa}
Let $\psi \in C_b^{\infty}$ and $F \in \mathbb D^{\infty}$. Then $\psi(F) \in \mathbb D^{\infty}$. Moreover, for any integer $k \geq 1$ and $p \geq 1$, 
\begin{equation*}
\| D^{k}(\psi(F)) \|_{L^p(\Omega;\mathcal H^{\otimes k})} 
\leq C(\|D^k F\|_{L^p(\Omega;\mathcal H^{\otimes k})} + \|F\|^k_{\mathbb D^{k-1,pk}} + 1)
\end{equation*}
holds for some $C>0$ depending on $k,p$ and $\psi$. 
\end{lemma}

\begin{lemma} \label{lem.Malliavin}
If $f \in C_{b}^{1}$, then for any $t \in [0,T]$, $x(t) \in \mathbb{D}^{1,\infty}$, and for a.e.\ $r \in [0,T]$,
\begin{align}\label{eq.DrXt}
D_r x(t) = \Big( \frac{1}{\Gamma(\alpha)} \int_{r}^t (t-s)^{\alpha-1 }f^{\prime}(x(s)) D_r x(s) \mathrm{d} s + \frac{\sigma}{\Gamma(\alpha)} (t-r)^{\alpha-1} \Big) \mathbf{1}_{[0,t)}(r).
\end{align}
Moreover, there exists some constant $C = C(\alpha,H,\sigma,T)$ such that for a.e.\ $r\in[0,T]$,
\begin{align}\label{eq.Drxt}
| D_r x(t) | \leq C (t - r)^{\alpha-1} \mathbf{1}_{[0,t)}(r),\quad \mbox{a.s.}
\end{align}
\end{lemma}

\begin{remark}
By revisiting the proof of Theorem 3.1 of \cite{DaiHongShengZhou2023}, the condition $f \in C_{b}^{1}$ is sufficient to the conclusions of Lemma \ref{lem.Malliavin}. 
\end{remark}

The subsequent lemma is extremely important for the proof of Theorem \ref{thm.DensityConvergence}. 

\begin{lemma} 
If $f \in C_{b}^{1}$, then there exists some positive constant $C = C(\alpha,H,\sigma,T)$ such that 
\begin{align} \label{lem.DxBound}
\sup_{0 \leq t \leq T} \| D x(t) \|_{\mathcal{H}} 
&\leq C,\quad \mbox{a.s.}, 
\end{align} 
and for any $0 < s < t \leq T$, 
\begin{align} \label{lem.DxRegularity} 
\| D x(t) - D x(s) \|_{\mathcal{H}} 
\leq C t^{H-\frac{1}{2}} s^{\frac{1}{2}-H} (t-s)^{\alpha+H-1},\quad \mbox{a.s.} 
\end{align} 
\end{lemma}

\begin{proof} 
For any $t \in [0,T]$, it follows from \eqref{eq.Drxt} that 
\begin{align*}
\| D x(t) \|_{\mathcal{H}}^2 &= \| K^{\ast} D x(t) \|_{L^2([0,T];\mathbb{R})}^2 = \int_0^T \Big| \int_u^T D_r x(t) \frac{\partial}{\partial r} \Phi_H (r,u) \rd r \Big|^2 \rd u \\
&= \int_0^T \Big| \int_u^T D_r x(t) c_H u^{\frac{1}{2}-H} (r-u)^{H-\frac{3}{2}} r^{H-\frac{1}{2}} \rd r \Big|^2 \rd u \\
&\leq C t^{2H-1} \int_0^t \Big| \int_u^t (t-r)^{\alpha-1} (r-u)^{H-\frac{3}{2}} \rd r \Big|^2 \, u^{1-2H} \rd u \\
&\leq C t^{2H-1} \int_0^t (t-u)^{2\alpha+2H-3} u^{1-2H} \rd u \\
&\leq C t^{2\alpha+2H-2} \leq C, 
\end{align*}
which shows that \eqref{lem.DxBound} holds. We next prove \eqref{lem.DxRegularity}. Set $0 < s < t \leq T$ and note that 
\begin{align}
&\quad\ \int_0^s \Big| (t-u)^{\alpha+H-\frac{3}{2}} - (s-u)^{\alpha+H-\frac{3}{2}} \Big|^2 u^{1-2H} \rd u \notag\\ 
&\leq C \int_0^{\frac{s}{2}} \Big| \int_s^t (v-u)^{\alpha+H-\frac{5}{2}} \rd v \Big|^2 u^{1-2H} \rd u 
+ \int_{\frac{s}{2}}^s \Big| (t-u)^{\alpha+H-\frac{3}{2}} - (s-u)^{\alpha+H-\frac{3}{2}} \Big|^2 u^{1-2H} \rd u \notag\\
&\leq C \int_0^{\frac{s}{2}} (s-\frac{s}{2})^{-1} \Big| \int_s^t (v-u)^{\alpha+H-2} \rd v \Big|^2 u^{1-2H} \rd u \notag\\
&\quad\ + C s^{1-2H} \int_0^s \Big| (t-u)^{\alpha+H-\frac{3}{2}} - (s-u)^{\alpha+H-\frac{3}{2}} \Big|^2 \rd u \notag\\ 
&\leq C s^{1-2H} (t-s)^{2\alpha+2H-2}. \label{eq.auxInteEsti}
\end{align} 
It follows from \eqref{eq.DrXt} that 
\begin{align*}
&\quad\ \| D x(t) - D x(s) \|_{\mathcal{H}} \\
&\leq \frac{1}{\Gamma(\alpha)} \Big\| \int_0^t (t-u)^{\alpha-1} f^{\prime}(x(u)) D x(u) \mathrm{d} u - \int_0^s (s-u)^{\alpha-1} f^{\prime}(x(u)) D x(u) \mathrm{d} u \Big\|_{\mathcal{H}} \\
&\quad\ + \frac{\sigma}{\Gamma(\alpha)} \Big\| (t-\cdot)^{\alpha-1} \mathbf{1}_{[0,t)}(\cdot) - (s-\cdot)^{\alpha-1} \mathbf{1}_{[0,s)}(\cdot) \Big\|_{\mathcal{H}} \\
&=: \frac{1}{\Gamma(\alpha)} \cI_1 + \frac{\sigma}{\Gamma(\alpha)} \cI_2. 
\end{align*} 
Firstly, by $f \in C_{b}^{1}$ and \eqref{lem.DxBound}, 
\begin{align*}
\cI_1
&\leq \Big\| \int_0^s \big( (t-u)^{\alpha-1} - (s-u)^{\alpha-1} \big) f^{\prime}(x(u)) D x(u) \mathrm{d} u \Big\|_{\mathcal{H}} \\
&\quad\ + \Big\| \int_s^t (t-u)^{\alpha-1 }f^{\prime}(x(u)) D x(u) \mathrm{d} u \Big\|_{\mathcal{H}} \\ 
&\leq C \int_0^s \big( (s-u)^{\alpha-1} - (t-u)^{\alpha-1} \big) \| D x(u) \|_{\mathcal{H}} \mathrm{d} u 
 + C \int_s^t (t-u)^{\alpha-1 } \| D x(u) \|_{\mathcal{H}} \mathrm{d} u \\ 
&\leq C (t-s)^{\alpha}. 
\end{align*}
Secondly, using \eqref{eq.auxInteEsti} shows 
\begin{align*}
\cI_2^2 
&= \int_0^T \Big| \int_u^T \Big( (t-r)^{\alpha-1} \mathbf{1}_{[0,t)}(r) - (s-r)^{\alpha-1} \mathbf{1}_{[0,s)}(r) \Big) c_H u^{\frac{1}{2}-H} (r-u)^{H-\frac{3}{2}} r^{H-\frac{1}{2}} \rd r \Big|^2 \rd u \\
&\leq c_H^2 t^{2H-1} \int_0^t \Big| \int_u^t \Big( (t-r)^{\alpha-1} \mathbf{1}_{[0,t)}(r) - (s-r)^{\alpha-1} \mathbf{1}_{[0,s)}(r) \Big) (r-u)^{H-\frac{3}{2}} \rd r \Big|^2 u^{1-2H} \rd u \\
&= c_H^2 t^{2H-1} \int_0^s \Big| \int_u^t (t-r)^{\alpha-1} (r-u)^{H-\frac{3}{2}} \rd r - \int_u^s (s-r)^{\alpha-1} (r-u)^{H-\frac{3}{2}} \rd r \Big|^2 u^{1-2H} \rd u \\
&\quad\ + c_H^2 t^{2H-1} \int_s^t \Big| \int_u^t (t-r)^{\alpha-1} (r-u)^{H-\frac{3}{2}} \rd r \Big|^2 u^{1-2H} \rd u \\
&\leq C t^{2H-1} \int_0^s \Big| (t-u)^{\alpha+H-\frac{3}{2}} - (s-u)^{\alpha+H-\frac{3}{2}} \Big|^2 u^{1-2H} \rd u \\
&\quad\ + C t^{2H-1} \int_s^t (t-u)^{2\alpha+2H-3} u^{1-2H} \rd u \\
&\leq C t^{2H-1} s^{1-2H} (t-s)^{2\alpha+2H-2}. 
\end{align*}
Thus, we obtain 
\begin{align*}
\| D x(t) - D x(s) \|_{\mathcal{H}} 
&\leq C (t-s)^{\alpha} + C t^{H-\frac{1}{2}} s^{\frac{1}{2}-H} (t-s)^{\alpha+H-1} \\
&\leq C t^{H-\frac{1}{2}} s^{\frac{1}{2}-H} (t-s)^{\alpha+H-1}. 
\end{align*} 
The proof is completed. 
\end{proof}

\section{Proofs of Theorems \ref{thm.ExactDensityExis}--\ref{thm.DensityConvergence}}
\label{sec4}

In this section, we provide the detailed proofs for Theorems \ref{thm.ExactDensityExis}--\ref{thm.DensityConvergence}.

\subsection{Proof of Theorem \ref{thm.ExactDensityExis}}

In view of Theorem 2.1.2 of \cite{Nualart2006} and Lemma \ref{lem.Malliavin}, it suffices to show that $\| D x(t) \|_{\mathcal{H}}$ is strictly positive a.s.\ for any $t \in (0,T]$. For any $t \in (0,T]$, we have
\begin{align*}
\| D x(t) \|_{\mathcal{H}}^2 &= \| K^{\ast} D x(t) \|_{L^2([0,T];\mathbb{R})}^2 = \int_0^T | \big( K^{\ast} D x(t) \big) (s) |^2 \rd s \\
&= \int_0^T \Big| \int_s^T D_r x(t) \frac{\partial K_H}{\partial r}(r,s) \rd r \Big|^2 \rd s \\
&= \int_0^T \Big| \int_s^T D_r x(t) c_H s^{\frac{1}{2}-H} (r-s)^{H-\frac{3}{2}} r^{H-\frac{1}{2}} \rd r \Big|^2 \rd s \\
&\geq \int_{t-\epsilon}^t \Big| \int_s^T \Big( \int_r^t (t-u)^{\alpha-1 }f^{\prime}(x(u)) D_r x(u) \mathrm{d} u + \sigma (t-r)^{\alpha-1} \Big) \mathbf{1}_{[0,t)}(r) \\
&\qquad\quad \times \frac{c_H}{\Gamma(\alpha)} s^{\frac{1}{2}-H} (r-s)^{H-\frac{3}{2}} r^{H-\frac{1}{2}} \rd r \Big|^2 \rd s,
\end{align*}
where the positive number $\epsilon \leq \frac{t}{2}$ can be arbitrarily small. Then, using the elementary inequality $(a+b)^2 \geq \frac{1}{2}a^2 - b^2$ (for $a, b \in \mathbb{R}$) obtains
\begin{align*}
&\quad\ \| D x(t) \|_{\mathcal{H}}^2 \\
&\geq \frac{c_H^2 \sigma^2}{2\Gamma^2(\alpha)} \int_{t-\epsilon}^t \Big| \int_s^T (t-r)^{\alpha-1} \mathbf{1}_{[0,t)}(r) s^{\frac{1}{2}-H} (r-s)^{H-\frac{3}{2}} r^{H-\frac{1}{2}} \rd r \Big|^2 \rd s \\
&\quad\ - \frac{c_H^2}{\Gamma^2(\alpha)} \int_{t-\epsilon}^t \Big| \int_s^T \int_r^t (t-u)^{\alpha-1 }f^{\prime}(x(u)) D_r x(u) \mathrm{d} u \mathbf{1}_{[0,t)}(r) s^{\frac{1}{2}-H} (r-s)^{H-\frac{3}{2}} r^{H-\frac{1}{2}} \rd r \Big|^2 \rd s \\
&=: \frac{c_H^2 \sigma^2}{2\Gamma^2(\alpha)} I_1 - \frac{c_H^2}{\Gamma^2(\alpha)} I_2.
\end{align*}
On the one hand, by $H \in (\frac{1}{2},1)$ and $\alpha \in (1-H,1)$, 
\begin{align*}
I_1 &\geq \int_{t-\epsilon}^t \Big| \int_s^t (t-r)^{\alpha-1} s^{\frac{1}{2}-H} (r-s)^{H-\frac{3}{2}} r^{H-\frac{1}{2}} \rd r \Big|^2 \rd s \\
&\geq (t-\epsilon)^{2H-1} \int_{t-\epsilon}^t \Big| \int_s^t (t-r)^{\alpha-1} (r-s)^{H-\frac{3}{2}} \rd r \Big|^2 s^{1-2H} \rd s \\
&\geq C t^{2H-1} \int_{t-\epsilon}^t (t-s)^{2\alpha+2H-3} s^{1-2H} \rd s \\
&\geq C \int_{t-\epsilon}^t (t-s)^{2\alpha+2H-3} \rd s \\
&\geq C \epsilon^{2\alpha+2H-2},
\end{align*}
where the positive constant $C$ is independent of $\epsilon$. On the other hand, using the assumption $f \in C_b^1$ and \eqref{eq.Drxt} yields
\begin{align*}
I_2 &\leq C \int_{t-\epsilon}^t \Big| \int_s^t \int_r^t (t-u)^{\alpha-1 } (u - r)^{\alpha-1} \mathrm{d} u \, s^{\frac{1}{2}-H} (r-s)^{H-\frac{3}{2}} r^{H-\frac{1}{2}} \rd r \Big|^2 \rd s \\
&\leq C \int_{t-\epsilon}^t \Big| \int_s^t (t-r)^{2\alpha-1} s^{\frac{1}{2}-H} (r-s)^{H-\frac{3}{2}} r^{H-\frac{1}{2}} \rd r \Big|^2 \rd s \\
&\leq C t^{2H-1} \int_{t-\epsilon}^t (t-s)^{4\alpha+2H-3} s^{1-2H} \rd s \\
&\leq C t^{2H-1} (t-\epsilon) ^{1-2H}\int_{t-\epsilon}^t (t-s)^{4\alpha+2H-3} \rd s \\
&\leq C \epsilon^{4\alpha+2H-2},
\end{align*}
where the positive constant $C$ is independent of $\epsilon$. Hence, one can conclude that there exist positive constants $C_1$ and $C_2$, which are independent of $\epsilon$, such that
\begin{align*}
\| D x(t) \|_{\mathcal{H}}^2
\geq C_1 \epsilon^{2\alpha+2H-2} - C_2
\epsilon^{4\alpha+2H-2}.
\end{align*}
Finally, one can choose $\epsilon \in (0, (\frac{C_1}{C_2})^{\frac{1}{2\alpha}} \wedge \frac{t}{2})$ such that $\| D x(t) \|_{\mathcal{H}}^2 > 0$. In particular, one can take $\epsilon = \frac{1}{2}(\frac{C_1}{C_2})^{\frac{1}{2\alpha}} \wedge \frac{t}{4}$ to show that there exists a positive constant $C_0$ satisfying
\begin{align} \label{eq.DxPositive}
\| D x(t) \|_{\mathcal{H}}^2 \geq C_0 > 0.
\end{align}
Thus, the proof is completed by Theorem 2.1.2 of \cite{Nualart2006}, Lemma \ref{lem.Malliavin} and \eqref{eq.DxPositive}. 
\hfill$\Box$

\subsection{Proof of Theorem \ref{thm.ExactDensitySmooth}}

In view of Theorem 2.1.4 of \cite{Nualart2006} and \eqref{eq.DxPositive}, it suffices to show that $x(t)\in \mathbb{D}^\infty$ for all $t\in[0,T]$. To this end, consider the following Picard iteration sequence
\begin{align}\label{eq.Picard}
x^{(n+1)}(t) = x_0 + \frac{1}{\Gamma(\alpha)} \int_0^t (t-s)^{\alpha-1} f \big( x^{(n)}(s) \big) \mathrm{d} s + \frac{\sigma}{\Gamma(\alpha)} \int_0^{t} (t-s)^{\alpha-1} \mathrm{d}W_H(s), \ \ n \geq 0
\end{align}
with $x^{(0)}(t) = x_0$ for all $t \in [0,T]$. Since $x^{(0)}(t)$ is deterministic, we have $x^{(0)}(t)\in \mathbb D^\infty$ for any $t\in[0,T]$. Assume by induction that $x^{(n)}(t)\in \mathbb D^\infty$ for any $t\in[0,T]$. Then by $f \in C_b^\infty$ and Lemma \ref{Faa}, we have that $f(x^{(n)}(s))\in \mathbb D^\infty$ for any $s\in[0,T]$, and thus $x^{(n+1)}(t)\in \mathbb D^\infty$. Hence, the above induction argument shows that $x^{(n)}(t)\in \mathbb D^\infty$ for all $n\in\mathbb N$.
It has been shown in Theorem 3.1 of \cite{DaiHongShengZhou2023} that $\{ x^{(n)}(t) \}_{n=0}^{\infty}$ converges to $x(t)$ in $L^p(\Omega;\mathbb{R})$ for any $p \geq 1$.

Based on Lemma 1.5.3 of \cite{Nualart2006}, we only need to show that for any $k\in\mathbb{N}_+$ and $p \geq 1$,
\begin{align}\label{eq:Dkx}
\sup_{n\ge0}\sup_{t\in[0,T]}\|D^kx^{(n)}(t)\|_{L^p(\Omega;\mathcal H^{\otimes k})}\le C,
\end{align}
in order to complete the proof of $x(t)\in \mathbb{D}^\infty$ for all $t\in[0,T]$. Next, we proceed to prove \eqref{eq:Dkx} by an induction argument on $k\in\mathbb N_+$. 
For $k=1$, \eqref{eq:Dkx} has been obtained in (3.9) of \cite{DaiHongShengZhou2023}. Now let $m\ge2$ be an integer and assume by induction that \eqref{eq:Dkx} holds for $k=0,1,\ldots,m-1$ and all $p \geq 1$. 
It suffices to prove that \eqref{eq:Dkx} holds for $k=m$ and all $p \geq 1$. Indeed, for arbitrarily fixed $p \geq 1$,
taking the $m$th ($m\ge2$) Malliavin derivative and then taking the $\|\cdot\|_{L^p(\Omega;\mathcal H^{\otimes m})}$-norm on both sides of \eqref{eq.Picard}, we obtain from Lemma \ref{Faa} that for any $n\in\mathbb N$,
\begin{align*}
&\quad\ \|D^mx^{(n+1)}(t) \|_{L^p(\Omega;\mathcal H^{\otimes m})} \\
&\leq C\int_0^t (t-s)^{\alpha-1} \left(\|D^m x^{(n)}(s) \|_{L^p(\Omega;\mathcal H^{\otimes m})} + C\|x^{(n)}(s)\|_{\mathbb D^{m-1,pm}}^m + C\right)\mathrm{d} s \\
&\leq C\int_0^t (t-s)^{\alpha-1} \|D^m x^{(n)}(s) \|_{L^p(\Omega;\mathcal H^{\otimes m})}\mathrm{d} s + C,
\end{align*}
where the last step is due to the induction assumption that \eqref{eq:Dkx} holds for $k=0,1,\ldots,m-1$ and all $p \geq 1$. Applying further Lemma 3.1 of \cite{DaiHongShengZhou2023} yields that for any $p \geq 1$,
$$\sup_{n\ge0}\sup_{t\in[0,T]}\|D^mx^{(n)}(t) \|_{L^p(\Omega;\mathcal H^{\otimes m})}<\infty,$$
as required. Hence we have proved that $x(t)\in\mathbb D^\infty$ for all $t\in[0,T]$. Finally, recalling Theorem 2.1.4 of \cite{Nualart2006} and \eqref{eq.DxPositive} completes the proof. 
\hfill$\Box$

\subsection{Proofs of Theorems \ref{thm.EulerDensityExis} and \ref{thm.EulerDensitySmooth}}

\underline{\emph{Proof of Theorem \ref{thm.EulerDensityExis}}}. Obviously, $x_0 \in \mathbb{D}^{1,\infty}$. Assume by induction that $x_i \in \mathbb{D}^{1,\infty}$ for $i = 0,1,\cdots,n$. Then, it follows from $f \in C_{b}^{1}$ that $x_{n+1} \in \mathbb{D}^{1,\infty}$. Thus, we have $x_n \in \mathbb{D}^{1,\infty}$ for all integer $n \in \{1,2,\cdots,N\}$. Moreover, the chain rule of the Malliavin derivative indicates that for $r < t_n$,
\begin{align} \label{eq.DrXn}
D_r x_n = \frac{1}{\Gamma(\alpha)} \sum_{j=1}^{n} f'(x_{j-1}) D_r x_{j-1} \mathbf 1_{[0,t_{j-1})}(r) \int_{t_{j-1}}^{t_j} (t_n-s)^{\alpha-1} \mathrm{d}s + \frac{\sigma}{\Gamma(\alpha)} (t_n-r)^{\alpha-1},
\end{align}
and for $r \in [t_n,T]$, $D_r x_n = 0$. Using the assumption $f \in C_{b}^{1}$ yields
\begin{align*}
| D_r x_n | \leq C \Big( \sum_{j=1}^{n} | f'(x_{j-1}) D_r x_{j-1} | \int_{t_{j-1}}^{t_j} (t_n-s)^{\alpha-1} \mathrm{d}s + (t_n-r)^{\alpha-1} \Big) \mathbf 1_{[0,t_n)}(r).
\end{align*}
Thus, the generalized Gr\"onwall inequality gives
\begin{align} \label{eq.DrXnUpper}
| D_r x_n | \leq C (t_n-r)^{\alpha-1} \mathbf 1_{[0,t_n)}(r), \quad \mbox{for } n=0,1,\cdots,N.
\end{align}

In view of Theorem 2.1.2 of \cite{Nualart2006}, it suffices to show that $\| D x_n \|_{\mathcal{H}}$ is strictly positive for any integer $n \in \{1,2,\cdots,N\}$. For any $t \in (0,T]$, we have
\begin{align*}
&\quad\ \| D x_n \|_{\mathcal{H}}^2 = \| K^{\ast} D x_n \|_{L^2([0,T];\mathbb{R})}^2 = \int_0^T | \big( K^{\ast} D x_n \big) (s) |^2 \rd s \\
&= \int_0^T \Big| \int_s^T D_r x_n \frac{\partial K_H}{\partial r}(r,s) \rd r \Big|^2 \rd s = \int_0^T \Big| \int_s^T D_r x_n c_H s^{\frac{1}{2}-H} (r-s)^{H-\frac{3}{2}} r^{H-\frac{1}{2}} \rd r \Big|^2 \rd s \\
&\geq \int_{t_n-\frac{h}{2}}^{t_n} \Big| \int_s^T \Big(
\sum_{j=1}^{n} f'(x_{j-1}) D_r x_{j-1} \mathbf 1_{[0,t_{j-1})}(r) \int_{t_{j-1}}^{t_j} (t_n-u)^{\alpha-1} \mathrm{d}u + \sigma (t_n-r)^{\alpha-1} \mathbf{1}_{[0,t_n)}(r) \Big) \\
&\qquad\qquad \times \frac{c_H}{\Gamma(\alpha)} s^{\frac{1}{2}-H} (r-s)^{H-\frac{3}{2}} r^{H-\frac{1}{2}} \rd r \Big|^2 \rd s \\
&= \frac{c_H^2 \sigma^2}{\Gamma^2(\alpha)} \int_{t_n-\frac{h}{2}}^{t_n} \Big| \int_s^T (t_n-r)^{\alpha-1} \mathbf{1}_{[0,t_n)}(r) s^{\frac{1}{2}-H} (r-s)^{H-\frac{3}{2}} r^{H-\frac{1}{2}} \rd r \Big|^2 \rd s.
\end{align*}
Then, using the assumptions $H \in (\frac{1}{2},1)$ and $\alpha \in (1-H,1)$ reveals that for some $\tilde{C} > 0$, 
\begin{align} \label{eq.DxnPositive}
\| D x_n \|_{\mathcal{H}}^2 
&\geq C \int_{t_n-\frac{h}{2}}^{t_n} \Big| \int_s^{t_n} (t_n-r)^{\alpha-1} s^{\frac{1}{2}-H} (r-s)^{H-\frac{3}{2}} r^{H-\frac{1}{2}} \rd r \Big|^2 \rd s \nonumber\\
&\geq C \int_{t_n-\frac{h}{2}}^{t_n} \Big| \int_s^{t_n} (t_n-r)^{\alpha-1} (r-s)^{H-\frac{3}{2}} \rd r \Big|^2 \rd s \nonumber\\
&= C \int_{t_n-\frac{h}{2}}^{t_n} (t_n-s)^{2\alpha+2H-3} \rd s \geq \tilde{C} h^{2\alpha+2H-2} > 0. 
\end{align}
Hence, the proof is completed by using Theorem 2.1.2 of \cite{Nualart2006}.
\hfill$\Box$

\vskip 0.8em 

\underline{\emph{Proof of Theorem \ref{thm.EulerDensitySmooth}}}. Obviously, $x_0 \in \mathbb{D}^{\infty}$. Assume by induction that $x_1,\ldots, x_n \in \mathbb{D}^{\infty}$. Then, it follows from $f \in C_{b}^{\infty}$ and Lemma \ref{Faa} that $x_{n+1} \in \mathbb{D}^{\infty}$. Thus, the proof is completed by recalling Theorem 2.1.4 of \cite{Nualart2006} and \eqref{eq.DxnPositive}. 
\hfill$\Box$

\subsection{Proof of Theorem \ref{thm.DensityConvergence}}

For any $n \in \{1,2,\cdots,N\}$, according to \eqref{eq.DrXt} and \eqref{eq.DrXn}, one has 
\begin{align*}
\Gamma(\alpha)(D x(t_n) - D x_n) 
&= \sum_{j=1}^{n} \int_{t_{j-1}}^{t_j} (t_n-s)^{\alpha-1} \big( f^{\prime}(x(s)) D x(s) - f'(x_{j-1}) D x_{j-1} \big) \mathrm{d}s \\
&= \sum_{j=1}^{n} \int_{t_{j-1}}^{t_j} (t_n-s)^{\alpha-1} \big( f^{\prime}(x(s)) - f^{\prime}(x(t_{j-1})) \big) D x(s) \mathrm{d}s \\
&\quad\ + \sum_{j=1}^{n} \int_{t_{j-1}}^{t_j} (t_n-s)^{\alpha-1} \big( f^{\prime}(x(t_{j-1})) - f'(x_{j-1}) \big) D x(s) \mathrm{d}s \\
&\quad\ + \sum_{j=1}^{n} \int_{t_{j-1}}^{t_j} (t_n-s)^{\alpha-1} f'(x_{j-1}) \big( D x(s) - D x(t_{j-1}) \big) \mathrm{d}s \\
&\quad\ + \sum_{j=1}^{n} \int_{t_{j-1}}^{t_j} (t_n-s)^{\alpha-1} f'(x_{j-1}) \big( D x(t_{j-1}) - D x_{j-1} \big) \mathrm{d}s. 
\end{align*}
Using $f \in C_{b}^{1}$, \eqref{lem.DxBound} and H\"older's inequality indicates 
\begin{align*} 
\mathbb{E}\big[ \| D x(t_n) - D x_n \|_{\mathcal{H}}^2 \big] 
 &\leq C \sum_{j=1}^{n} \int_{t_{j-1}}^{t_j} (t_n-s)^{\alpha-1} \mathbb{E}\big[ | f^{\prime}(x(s)) - f^{\prime}(x(t_{j-1})) |^2 \big] \mathrm{d}s \\
&\quad\ + C \sum_{j=1}^{n} \int_{t_{j-1}}^{t_j} (t_n-s)^{\alpha-1} \mathbb{E}\big[ | f^{\prime}(x(t_{j-1})) - f'(x_{j-1}) |^2 \big] \mathrm{d}s \\
&\quad\ + C \sum_{j=1}^{n} \int_{t_{j-1}}^{t_j} (t_n-s)^{\alpha-1} \mathbb{E}\big[ \| D x(s) - D x(t_{j-1}) \|_{\mathcal{H}}^2 \big] \mathrm{d}s \\
&\quad\ + C \sum_{j=1}^{n} \int_{t_{j-1}}^{t_j} (t_n-s)^{\alpha-1} \mathbb{E}\big[ \| D x(t_{j-1}) - D x_{j-1} \|_{\mathcal{H}}^2 \big] \mathrm{d}s. 
\end{align*} 
By $f \in C_{b}^{2}$ and \cite[Lemma 2.3 and Proposition 3.1]{FangLi2020}, one can read 
\begin{align*} 
\mathbb{E}\big[ \| D x(t_n) - D x_n \|_{\mathcal{H}}^2 \big] 
 &\leq C h^{2\alpha+2H-2} + C h^{2\alpha+2H-2} + C \int_0^{t_1} (t_n-s)^{\alpha-1} \mathbb{E}\big[\| D x(s) \|_{\mathcal{H}}^2 \big] \mathrm{d}s \\
&\quad\ + C \sum_{j=2}^{n} \int_{t_{j-1}}^{t_j} (t_n-s)^{\alpha-1} \mathbb{E}\big[ \| D x(s) - D x(t_{j-1}) \|_{\mathcal{H}}^2 \big] \mathrm{d}s \\
&\quad\ + C \sum_{j=1}^{n} \int_{t_{j-1}}^{t_j} (t_n-s)^{\alpha-1} \mathbb{E}\big[ \| D x(t_{j-1}) - D x_{j-1} \|_{\mathcal{H}}^2 \big] \mathrm{d}s. 
\end{align*} 
It follows from the proof of \eqref{lem.DxBound} that $\mathbb{E} [\| D x(s) \|_{\mathcal{H}}^2] \leq C h^{2\alpha+2H-2}$ for all $s \in [0,t_1]$, which together with \eqref{lem.DxRegularity} implies 
\begin{align*} 
\mathbb{E}\big[ \| D x(t_n) - D x_n \|_{\mathcal{H}}^2 \big] 
 &\leq C h^{2\alpha+2H-2} + C h^{2\alpha+2H-2} \sum_{j=2}^{n} \int_{t_{j-1}}^{t_j} (t_n-s)^{\alpha-1} s^{2H-1} t_{j-1}^{1-2H} \mathrm{d}s \\
&\quad\ + C \sum_{j=1}^{n} \int_{t_{j-1}}^{t_j} (t_n-s)^{\alpha-1} \mathbb{E}\big[ \| D x(t_{j-1}) - D x_{j-1} \|_{\mathcal{H}}^2 \big] \mathrm{d}s \\
&\leq C h^{2\alpha+2H-2} + C \sum_{j=1}^{n} \int_{t_{j-1}}^{t_j} (t_n-s)^{\alpha-1} \mathbb{E}\big[ \| D x(t_{j-1}) - D x_{j-1} \|_{\mathcal{H}}^2 \big] \mathrm{d}s, 
\end{align*} 
where the last step used the inequality 
\begin{align*}
\sum_{j=2}^{n} \int_{t_{j-1}}^{t_j} (t_n-s)^{\alpha-1} s^{2H-1} t_{j-1}^{1-2H} \mathrm{d}s 
\leq C \int_0^{t_n} (t_n-s)^{\alpha-1} s^{2H-1} s^{1-2H} \mathrm{d}s 
\leq C. 
\end{align*}
Here, the fact $t_{j-1} \geq \frac{1}{2}t_j$ ($j \geq 2$) is used for the first inequality. Then, applying the Gronwall-type inequality (e.g., Lemma 3.1 of \cite{DaiHongSheng2023}) yields
\begin{align*}
\mathbb{E}\big[ \| D x(t_n) - D x_n \|_{\mathcal{H}}^2 \big] 
\leq C h^{2\alpha+2H-2}.
\end{align*}
Finally, using \cite[(6.1)]{HongJinSheng2023}, Proposition \ref{prop.D12} and \cite[Proposition 3.1]{FangLi2020} shows 
\begin{align*}
\| q_{x(t_n)} - q_{x_n} \|_{L^1(\mathbb{R};\mathbb{R})} 
&= d_{\mathrm{TV}} \left( x(t_n), x_n \right) \leq C \| x(t_n) - x_n \|_{\mathbb D^{1,2}} \\
&= C \big( \mathbb{E}\big[ |x(t_n) - x_n|^2 \big] + \mathbb{E}\big[ \| D x(t_n) - D x_n \|_{\mathcal{H}}^2\big] \big)^{\frac{1}{2}} \\
&\leq C h^{\alpha+H-1}.
\end{align*}
The proof is completed. 
\hfill$\Box$

\section*{Acknowledgments}

This work is supported by National Natural Science Foundation of China (No.\ 12201228), China Postdoctoral Science Foundation (No.\ 2022M713313), and the Fundamental Research Funds for the Central Universities (No.\ 3034011102).

\section*{Declaration of competing interest}
The authors declare that they have no conflict of interest.

\section*{Data availability}
Data sharing is not applicable to this article as no datasets were generated or analyzed during the current study.

\section*{CRediT authorship contribution statement}
\textbf{Xinjie Dai:} Writing - original draft, Writing - review \& editing. \textbf{Diancong Jin:} Writing - original draft, Writing - review \& editing.

\bibliographystyle{amsplain}
\bibliography{RefBib}

\providecommand{\bysame}{\leavevmode\hbox to3em{\hrulefill}\thinspace}
\providecommand{\MR}{\relax\ifhmode\unskip\space\fi MR }
\providecommand{\MRhref}[2]{%
  \href{http://www.ams.org/mathscinet-getitem?mr=#1}{#2}
}
\providecommand{\href}[2]{#2}
\begin{thebibliography}{10}

\bibitem{BesaluMarquez2021}
M.~Besal\'{u}, D.~M\'{a}rquez-Carreras, and E.~Nualart, \emph{Existence and
  smoothness of the density of the solution to fractional stochastic integral
  {V}olterra equations}, Stochastics \textbf{93} (2021), no.~4, 528--554.

\bibitem{ChenCuiHongSheng2023}
C.~Chen, J.~Cui, J.~Hong, and D.~Sheng, \emph{Accelerated exponential {E}uler
  scheme for stochastic heat equation: convergence rate of the density}, IMA J.
  Numer. Anal. \textbf{43} (2023), no.~2, 1181--1220.

\bibitem{DaiHongSheng2023}
X.~Dai, J.~Hong, and D.~Sheng, \emph{Error analysis of numerical methods on
  graded meshes for stochastic {V}olterra equations}, Preprint (2023), arXiv:\
  2308.16696.

\bibitem{DaiHongShengZhou2023}
X.~Dai, J.~Hong, D.~Sheng, and T.~Zhou, \emph{Strong error analysis of {E}uler
  methods for overdamped generalized {L}angevin equations with fractional
  noise: nonlinear case}, ESAIM Math. Model. Numer. Anal. \textbf{57} (2023),
  no.~4, 1981--2006.

\bibitem{DaiXiao2021}
X.~Dai and A.~Xiao, \emph{A note on {E}uler method for the overdamped
  generalized {L}angevin equation with fractional noise}, Appl. Math. Lett.
  \textbf{111} (2021), 106669.

\bibitem{DidierNguyen2022}
G.~Didier and H.~D. Nguyen, \emph{The generalized {L}angevin equation in
  harmonic potentials: anomalous diffusion and equipartition of energy}, Comm.
  Math. Phys. \textbf{393} (2022), no.~2, 909--954.

\bibitem{Fan2015}
X.~Fan, \emph{Stochastic {V}olterra equations driven by fractional {B}rownian
  motion}, Front. Math. China \textbf{10} (2015), no.~3, 595--620.

\bibitem{FangLi2020}
D.~Fang and L.~Li, \emph{Numerical approximation and fast evaluation of the
  overdamped generalized {L}angevin equation with fractional noise}, ESAIM
  Math. Model. Numer. Anal. \textbf{54} (2020), no.~2, 431--463.

\bibitem{FriesenJin2024}
M.~Friesen and P.~Jin, \emph{Volterra square-root process: stationarity and
  regularity of the law}, Ann. Appl. Probab. \textbf{34} (2024), no.~1A,
  318--356.

\bibitem{HongHuang2020}
J.~Hong, C.~Huang, M.~Kamrani, and X.~Wang, \emph{Optimal strong convergence
  rate of a backward {E}uler type scheme for the {C}ox--{I}ngersoll--{R}oss
  model driven by fractional {B}rownian motion}, Stochastic Process. Appl.
  \textbf{130} (2020), no.~5, 2675--2692.

\bibitem{HongJinSheng2023}
J.~Hong, D.~Jin, and D.~Sheng, \emph{Density convergence of a fully discrete
  finite difference method for stochastic {C}ahn--{H}illiard equation}, Math.
  Comp. DOI:\ doi.org/10.1090/mcom/3928 (2023).

\bibitem{Kou2008}
S.~C. Kou, \emph{Stochastic modeling in nanoscale biophysics: subdiffusion
  within proteins}, Ann. Appl. Stat. \textbf{2} (2008), no.~2, 501--535.

\bibitem{LeiBakerLi2016}
H.~Lei, N.~A. Baker, and X.~Li, \emph{Data-driven parameterization of the
  generalized {L}angevin equation}, Proc. Natl. Acad. Sci. USA \textbf{113}
  (2016), no.~50, 14183--14188.

\bibitem{LiLiu2019}
L.~Li and J.-G. Liu, \emph{A discretization of {C}aputo derivatives with
  application to time fractional {SDE}s and gradient flows}, SIAM J. Numer.
  Anal. \textbf{57} (2019), no.~5, 2095--2120.

\bibitem{LiLiu2017}
L.~Li, J.-G. Liu, and J.~Lu, \emph{Fractional stochastic differential equations
  satisfying fluctuation-dissipation theorem}, J. Stat. Phys. \textbf{169}
  (2017), no.~2, 316--339.

\bibitem{NourdinPeccati2012}
I.~Nourdin and G.~Peccati, \emph{Normal {A}pproximations with {M}alliavin
  {C}alculus:\ {F}rom {S}tein's {M}ethod to {U}niversality}, Cambridge Tracts
  in Mathematics, Cambridge University Press, Cambridge, 2012.

\bibitem{NourdinPoly2013}
I.~Nourdin and G.~Poly, \emph{Convergence in total variation on {W}iener
  chaos}, Stochastic Process. Appl. \textbf{123} (2013), no.~2, 651--674.

\bibitem{Nualart2006}
D.~Nualart, \emph{{T}he {M}alliavin {C}alculus and {R}elated {T}opics}, second
  ed., Probability and its Applications (New York), Springer-Verlag, Berlin,
  2006.

\bibitem{VojtaSkinner2019}
T.~Vojta, S.~Skinner, and R.~Metzler, \emph{Probability density of the
  fractional {L}angevin equation with reflecting walls}, Phys. Rev. E
  \textbf{100} (2019), no.~4, 042142.

\bibitem{VojtaWarhover2021}
T.~Vojta and A.~Warhover, \emph{Probability density of fractional {B}rownian
  motion and the fractional {L}angevin equation with absorbing walls}, J. Stat.
  Mech. Theory Exp. (2021), no.~3, 033215.

\bibitem{XieCarE2024}
P.~Xie, R.~Car, and W.~E, \emph{Ab initio generalized {L}angevin equation},
  Proc. Natl. Acad. Sci. USA \textbf{121} (2024), no.~14, e2308668121.

\end{thebibliography}

\end{spacing}

\end{document}